\documentclass [12pt, reqno]{amsart}

\usepackage{amssymb,latexsym,amsmath,amsfonts, eucal}

\newtheorem{thm}{Theorem}[section]

\newtheorem{lem}[thm]{Lemma}
\newtheorem{prop}[thm]{Proposition}
\theoremstyle{definition}

\theoremstyle{remark}
\newtheorem{rem}[thm]{Remark}
\numberwithin{equation}{section}

\newcommand{\beas}{\begin{eqnarray*}}
\newcommand{\eeas}{\end{eqnarray*}}
\newcommand{\bes} {\begin{equation*}}
\newcommand{\ees} {\end{equation*}}
\newcommand{\be} {\begin{equation}}
\newcommand{\ee} {\end{equation}}
\newcommand{\bea} {\begin{eqnarray}}
\newcommand{\eea} {\end{eqnarray}}

\newcommand{\txt} {\textmd}

\begin{document}
\title[Support Theorems]{Support theorems on $\mathbb R^n$ and non-compact symmetric spaces}

\author[E. K. Narayanan]{E. K. Narayanan}

\author[AMIT Samanta]{AMIT Samanta}
\address{Department of Mathematics, Indian Institute of Science, Banaglore 560012}

\email[E. K. Narayanan]{naru@math.iisc.ernet.in} \email[Amit
Samanta]{amit@math.iisc.ernet.in}

\thanks{The first author was supported in part by a grant from UGC via DSA-SAP and the second author was supported by Research Fellowship of Indian Institiute of Science, Bangalore.}

\begin{abstract}
We consider convolution equations of the type $f*T=g$ where
$f,g\in L^{p}(\mathbb{R}^{n})$ and $T$ is a compactly supported
distribution. Under natural assumptions on the zero set of the
Fourier transform of $T$ we show that $f$ is compactly supported,
provided $g$ is. Similar results are proved for non compact
symmetric spaces as well. \vspace*{0.1in}

\begin{flushleft}
\vspace*{0.1in}

\end{flushleft}

\end{abstract}

\maketitle

\section{Introduction}

Support theorems have attained a lot of attention in the past. We
recall two such results. First, the famous result due to Helgason
\cite{He}. This result states the following: If a measurable
function $f$ on $\mathbb{R}^{n}$ satisfies $(1+|x|)^{N}f\in
L^{1}(\mathbb{R}^{n})$ for all $N > 0 $ and $f$ integrates to zero
over all spheres enclosing a fixed ball of radius $R>0$, then $f$
is supported in $B_{R}$, where $B_{R}$ is the ball of radius $R$
centred at the origin. An analogue holds also for rank one
symetric spaces of non compact type \cite{B}. The second is a
result by A. Sitaram . In \cite{S}, he proved the following
support theorem: If $f\in L^{1}(\mathbb{R}^{n})$ is such that
$f*\chi_{B_{r}}=g$, where $\chi_{B_r}$ is the indicator function
of $B_r$ and $g$ is supported in $B_{R}$, then supp $f\subseteq
B_{R+r}$.

In this paper we are interested in the second result. We consider
convolution equations of the form $f*T =g,$ where $T$ is a
compactly supported distribution on $\mathbb{R}^{n}$ and $f\in
L^{p}(\mathbb{R}^{n})$. The question we are interested in is: can
we conclude that $f$ is compactly supported, if $g$ is compactly
supported? Combining methods from several complex variables and
harmonic analysis we prove general support theorems under natural
assumptions on the zero set of the entire function $\widehat{T}$
(Fourier transform of $T$). When $T=\chi_{B_{r}}$ or $\mu_{r}$
(the normalized surface measure on the sphere of radius $r$ on
$\mathbb{R}^{n}$), this problem was studied by Sitaram \cite{S},
Volchkov \cite{V} etc. When $T$ is a distribution supported at the
origin, this becomes a problem in PDE. In \cite{T}, Treves proved
that, if $P(D)u=v$ and $v$ is compactly supported, then $u$ is
also compactly supported, provided
$u\in\mathcal{S}(\mathbb{R}^{n})$ (Schwartz space) and the variety
of zeros of each irreducible factor of $P$ in $\mathbb{C}^{n}$
intersects $\mathbb{R}^{n}$. These questions were later taken up
by Littman in \cite{L1} and \cite{L2}. Considering the principal
value integral
$$\int_{\mathbb{R}^{n}}\frac{\hat{v}(y)}{P(y)}e^{ix\cdot y}dy$$
he was able to show that $u$ is compactly supported with the
assumption that $\{x\in\mathbb{R}^{n}:P(x)=0\}$ has dimenssion
$(n-1)$. Hormander strengthened these results in \cite{Ho}. Our
results may be viewed as generalizations of these results. We end
this section with the following theorem from \cite{AN} which will
be needed later.

\begin{thm}
If $f\in L^{p}(\mathbb{R}^{n})$ and supp $\hat{f}$ is carried by
a $C^{1}$ manifold of dimension $d<n$ then $f=0$ provided $1\leq
p\leq \frac{2n}{d}$ and $d>0$. If $d=0$ then $f=0$ for $1\leq
p<\infty$.
\end{thm}

\section{Support theorems on $\mathbb{R}^{n}$}

In this section we prove support theorems on $\mathbb{R}^{n}$
under natural assumptions on the zero set of the Fourier
transform of the distribution $T$. Before we state our results we
recall some notation from several complex variables which will be
used through out.

Let $F$ be an entire function on $\mathbb{C}^{n}$. Then $Z_{F}$
will denote the zero set of $F$, ie $Z_{F}=\{z\in \mathbb{C}^{n}:
F(z)=0\}$.  The set $Z_{F}$ is a complex analytic set and the
connected components of $Z_{F}$ are precisely the irreducible
components of $Z_{F}$. For more details on complex analytic sets
we refer to \cite{C}. Let Reg $(Z_{F})$ denote the regular points
of $Z_{F}$. If $z\in Z_{F}$ then Ord$_{z}F$ will denote the order
of $F$ at $z$ (see \cite{C} page 16). We also recall that the
order is a constant on each connected component of Reg $(Z_{F})$.
If $A$ is a complex analytic set, Sing $A$ will denote the
singular points. That is, Sing $A=A-$Reg $A$.

We start with the following general result.

\begin{thm}
Let $T$ be a compactly supported distribution on $\mathbb{R}^{n}$
and $f\in L^{p}(\mathbb{R}^{n})$ for some $p$ with $1\leq
p\leq\frac{2n}{n-1}$. Assume the following:
\begin{itemize}
\item [\textbf{(a)}]If $V$ is any irreducible component of
$Z_{\widehat{T}}$, then dim$_{\mathbb{R}}(V\cap
\mathbb{R}^{n})=n-1$. \item [\textbf{(b)}] grad $\widehat{T}\neq
0$ on Reg $(Z_{\widehat{T}}) \cap \mathbb{R}^{n}$.
\end{itemize}
Suppose $f*T=g$, where $g$ is compactly supported, then $f$ is
also compactly supported.
\end{thm}

We need several lemmas for the proof of this theorem.

\begin{lem}
If $f\in L^{p}(\mathbb{R}^{n}),$ $p=\frac{2n}{n-1,}$ then
$\exists$ $r_{k}\rightarrow\infty$ such that, for any fixed
constants $s_{1},s_{2}>0$ we have $$\int_{r_{k}-s_{1}\leq|x|\leq
r_{k}+s_{2}}|f(x)|^{2}dx\rightarrow 0$$ as $k\rightarrow \infty$.
\end{lem}

\begin{proof}
By contrary, assume that $\exists $ $a>0$ and $R>0$ such that
\bea\label{eqn0}\int_{r-s_{1}\leq|x|\leq r+s_{2}}|f(x)|^{2}dx\geq
a,~~~~\forall r\geq R.\eea By Holder's inequality we have
$$\int_{r-s_{1}\leq|x|\leq r+s_{2}}|f(x)|^{2}dx\leq
\left(\int_{r-s_{1}\leq|x|\leq
r+s_{2}}|f(x)|^{\frac{2n}{n-1}}dx\right)^{\frac{n-1}{n}}\left(\int_{r-s_{1}\leq|x|\leq
r+s_{2}}dx\right)^{\frac{1}{n}}.$$ From (\ref{eqn0}) and the above
it follows that for some constant $c>0$
$$\int_{r-s_{1}\leq|x|\leq
r+s_{2}}|f(x)|^{\frac{2n}{n-1}}dx\geq\frac{c}{r},~~~~\forall
r>R.$$ Integrating with respect to $r$ and noting that $f\in
L^{p}(\mathbb{R}^{n})$, $p=\frac{2n}{n-1}$ we obtain a
contradiction. Hence the lemma is proved.
\end{proof}

\begin{lem}
Let $F$ and $G$ be two entire functions on $\mathbb{C}^{n}$ such
that

\indent {\rm{\textbf{(a)}}} Each connected component of Reg
$Z_{F}$ intersected with $\mathbb{R}^{n}$ has real dimension
\noindent $(n-1)$.

\indent {\rm{\textbf{(b)}}} (Reg
$Z_{F})\cap\mathbb{R}^{n}\subseteq Z_{G}\cap\mathbb{R}^{n}$.

\indent {\rm{\textbf{(c)}}} Ord$_{x}F\leq$Ord$_{x}G
~~~~\hspace*{0.1in} \forall ~~~ x\in\mathbb{R}^{n}\cap$ Reg
$Z_{F}$.

\noindent
 Then $\frac{G}{F}$ is an entire function.
\end{lem}

\begin{proof}
Let Reg $Z_{F}=\bigcup_{j\in J}S_{j}$ be the decomposition of Reg
$Z_{F}$ into connected components. Then $Z_{F}=\bigcup_{j\in
J}A_{j}$ where $A_{j}=\overline{S_{j}}$ gives the decomposition of
$Z_{F}$ into irreducible components. If the complex dimension
dim$_{\mathbb{C}}(A_{j}\cap Z_{G})\leq (n-2)$, then
dim$_{\mathbb{R}}(A_{j}\cap Z_{G}\cap\mathbb{R}^{n})\leq (n-2)$
which contradicts \textbf{(a)} due to \textbf{(b)} in the
assumptions. It follows that dim$_{\mathbb{C}}(A_{j}\cap
Z_{G})=(n-1)$. Since $A_{j}$ is an irreducible analytic set in
$\mathbb{C}^{n}$, this will force $A_{j}$ to be an irreducible
component of $Z_{G}$ (see \cite{C}). It follows that Reg ($Z_{F}$)
$\subseteq$ Reg ($Z_{G}$). Since the order is a constant on the
regular part of an analytic set we also have Ord$_{z}F\leq$
Ord$_{z}G$ ~$\forall$ $z\in $Reg $F$. Consequently $\frac{G}{F}$
is holomorphic in $\mathbb{C}^{n}-$ Sing $(Z_{F})$. However the
$(2n-2)$ Hausdroff measure of (Sing $Z_{F}$) is zero (see
\cite{C}) and so by Proposition 2,  page 298, in \cite{C},
$\frac{G}{F}$ extends to an entire function.
\end{proof}

\begin{lem}
Let $f\in L^{p}(\mathbb{R}^{n})$, $1\leq p\leq\frac{2n}{n-1}$.
Let $T$ be a compactly supported distribution on $\mathbb{R}^{n}$
and $f*T=g$, where $g$ is compactly supported. If $\widehat{T}$
is zero on a smooth $(n-1)$ dimensional manifold
$M\subseteq\mathbb{R}^{n}$ then $\widehat{g}(x)=0~~~\forall x\in
M$.
\end{lem}

\begin{proof}
By convolving with radial aproximate identities we may assume
that $f\in L^{p_{0}}(\mathbb{R}^{n})\cap
C^{\infty}(\mathbb{R}^{n})$ where $p_{0}=\frac{2n}{n-1}$ and
$T\in L^{1}(\mathbb{R}^{n})$. Let supp $T\subseteq B_{R_{1}}$ and
supp $g\subseteq B_{R_{2}}.$ For $r>0$ define
$f_{r}(x)=\chi_{|x|\leq r}(x)f(x)$ and write \bea
\label{eqn1}f_{r}*T=g+g_{r}.\eea If $r$ is very large then supp
$g_{r}\subseteq \{x:r-R_{1}\leq|x|\leq r+R_{1}\}$ and \bea
\label{eqn2}|g_{r}(x)|\leq|T*f_{r-2R_{1},r+2R_{1}}(x)|\eea where
$$f_{r-2R_{1},r+2R_{1}}(x)=\chi_{r-2R_{1}\leq|x|\leq
r+2R_{1}}(x)f(x).$$ Next, let $\phi\in
C_{c}^{\infty}(\mathbb{R}^{n})$ and consider the measure $\mu$
defined by $$d\mu=\phi(x)dx_{M}$$ where $dx_{M}$ is the surface
measure on $M$. Then $\mu$ is a compactly supported measure on
$M$.  Since $\widehat{T}$ is zero on $M$, it is easy to see by
taking the Fourier transform that $T*f_{r}*\widehat{\mu}$ vanishes
identically.

From (\ref{eqn1}) it follows that
$$g*\widehat{\mu}+g_{r}*\widehat{\mu}\equiv 0.$$ We will show
that $g_{r}*\widehat{\mu}(x)$ goes to zero $\forall x\in
\mathbb{R}^{n}$ as $r\rightarrow\infty$, which implies that
$g*\hat{\mu}$ vanishes idetically. Taking the Fourier transform
again we obtain that $\hat{g}$ vanishes on supp $\phi\cap M$.
Since $\phi$ was arbitrary this proves the lemma.

Fix $x_{0}\in\mathbb{R}^{n}$ and consider
$g_{r}*\widehat{\mu}(x_{0})$. We have, by (\ref{eqn2}) \bea
\label{eqn3}~~~|g_{r}*\widehat{\mu}(x_{0})|\leq \int_{r-R_{1}\leq
|y|\leq
r+R_{1}}|T*f_{r-2R_{1},r+2R_{1}}(y)||\widehat{\mu}(x_{0}-y)|dy.\eea
Now if $\nu$ is a compactly supported smooth measure on $M$ then
$$\left(\int_{S^{n-1}}|\widehat{\nu}(s\omega)|^{2}d\omega\right)^{\frac{1}{2}}\leq\frac{c}{s^{\frac{n-1}{2}}}, ~s> 0.$$
(See \cite{BGG}, Proposition 1, page 2563). Apply the above to the
measure $e^{ix_{0}\cdot y}\phi(y)dy_{M}$ on $M$ to obtain
$$\left(\int_{S^{n-1}}|\widehat{\mu}(x_{0}-s\omega)|^{2}d\omega\right)^{\frac{1}{2}}\leq\frac{c(x_{0})}{s^{\frac{n-1}{2}}}.$$
It follows that $$ \int_{r- R_1 \leq |y| \leq r+R_1} | \hat{\mu}
(x_0 -y) |^2~dy \leq C.$$ A simple application of the
Cauchy-Schwarz inequality to (\ref{eqn3}) along with the above
estimates give us
$$|g_{r}*\widehat{\mu}(x_{0})|\leq C(x_0)
||T||_{1}\left(\int_{r-R_{1}\leq|y|\leq
r+R_{1}}|f(y)|^{2}dy\right)^{\frac{1}{2}} .$$ Choosing $\{r_{k}\}$
as in Lemma \textbf{2.2} we finish the proof.
\end{proof}

\begin{proof}[Proof of Theorem {\rm{\textbf{2.1}}}] Without loss of generality we may assume that
$f\in L^{p_{0}}(\mathbb{R}^{n})$, $p_{0}=\frac{2n}{n-1}$. Since
$f*T=g$ and (Reg $Z_{\widehat{T}}$)$\cap\mathbb{R}^{n}$ is a
smooth $(n-1)$ dimensional manifold, Lemma \textbf{2.4} implies
that $\widehat{g}(x)=0$ if $\widehat{T}(x)=0$. Since grad
$\widehat{T}$ is non zero on Reg $Z_{\widehat{T}}$ we have Ord
$_{x}\widehat{T}=1$ if $x\in$ Reg $Z_{\widehat{T}}$. Since
$\widehat{g}(x)=0~~~\forall x\in (\txt {Reg}
Z_{\widehat{T}})\cap\mathbb{R}^{n}$ it follows that
Ord$_{x}\widehat{g}\geq$ Ord$_{x}\widehat{T}~~~\forall x\in$ Reg
$Z_{\widehat{T}}\cap\mathbb{R}^{n}$. By Lemma \textbf{2.3} we have
that $\frac{\widehat{g}}{\widehat{T}}$ is an entire function.
Hence we have \bea \label{eqn
i}\widehat{f}=\frac{\widehat{g}}{\widehat{T}}+\delta.\eea Where
$\delta$ is a distribution supported on
$Z_{\widehat{T}}\cap\mathbb{R}^{n}$. We will show that $\delta
\equiv 0.$ Let $\phi\in C_{c}^{\infty}(\mathbb{R}^{n})$.
Multiplying (\ref{eqn i}) with $\phi$ and taking the inverse
Fourier transform we obtain
$${(\phi \delta)}^{\check{}}=\check{\phi}*f-h$$ where
$h\in\mathcal{S}(\mathbb{R}^{n})$. Notice that $\check{\phi}*f\in
L^{p_{0}}(\mathbb{R}^{n})$, $p_{0}=\frac{2n}{n-1}$. From Theorem
\textbf{1.1} it follows that $\phi \delta=0$. Since $\phi$ was
arbitrary it follows that $\hat{f}=\frac{\hat{g}}{\hat{T}}.$ By
Malgrange's theorem $\hat{f}$ is an entire function of
exponential type. If $\hat{T}$ is slowly decreasing this readily
implies that $f$ is compactly supported. However this extra
assumption is not needed as can be seen below. Let
$\psi\in\mathcal{S}(\mathbb{R}^{n})$ be such that $\hat{\psi}$ is
compactly supported. Then \beas \widehat{(\psi
f)}(x)&=&\hat{\psi}*\hat{f}(x)\\&=&
\int_{\mathbb{R}^{n}}\hat{\psi}(t)\hat{f}(x-t)dt,\eeas clearly
extends to an entire function of exponential type. Since $\psi
f\in L^{1}(\mathbb{R}^{n})$, $\widehat{(\psi f)}$ is bounded on
$R^{n}$. By the Paley-Wiener theorem we obtain that $\psi f$ is
compactly supported which finishes the proof.
\end{proof}

\begin{rem}
It is possible to weaken the condition grad $\hat{T}\neq 0$ on Reg
$Z_{\hat{T}}\cap\mathbb{R}^{n}$ as follows. Let $V$ be any global
irreducible component of $Z_{\hat{T}}$. Then there exist an entire
function $f_{V}$ whose zero locus is exactly $V$ and there exits a
positive integer $k$ such that $\frac{\hat{T}}{f_{V}^{k}}$ is non
zero on $V$. This is an application of Cousin $\textbf{II}$
problem on $\mathbb{C}^{n}$. See \cite{GH}. This function $f_{V}$
is unique upto multiplication by units. A close examination of the
proof shows that it suffices to assume that grad $f_{V}\neq 0$ on
$V \cap \mathbb{R}^n$ for all $V$. In particular when
$\hat{T}=f_{1}^{m_{1}}f_{2}^{m_{2}}\cdot\cdot\cdot f_{k}^{m_{k}}$
where $f_{1}$, $f_{2}$, $\cdot\cdot\cdot, f_{k}$ are irreducible
entire functions then it suffices to assume that grad $f_{j}\neq
0$ on $Z_{f_{j}}\cap\mathbb{R}^{n}$. Also see Hormander \cite{HO}
Theorem 3.1.
\end{rem}

Next we show that if $1\leq p\leq 2$ or $T$ is a radial
distribution then the condition on grad $\hat{T}$ is not needed
in the Theorem \textbf{2.1}.

\begin{thm}
Let $1\leq p\leq 2$ and $f\in L^{p}(\mathbb{R}^{n})$. If $f*T$ is
compactly supported and condition \textbf{(a)} of the previous
theorem is satisfied then $f$ is compactly supported.
\end{thm}

\begin{proof}
Let $f*T=g$. Convolving with compactly supported approximate
identities we may assume that $f\in L^{2}(\mathbb{R}^{n})$ and
$g\in C_{c}^{\infty}$. Since $\hat{T}\hat{f}=\hat{g}$ and $f\in
L^{2}(\mathbb{R}^{n})$ we have
$\int_{\mathbb{R}^{n}}\left|\frac{\hat{g}}{\hat{T}}\right|^{2}<\infty$.
We will show that, if $x_{0}\in$Reg
$(Z_{\hat{T}})\cap\mathbb{R}^{n}$ then
Ord$_{x_{0}}(\hat{T})\leq$Ord$_{x_{0}}(\hat{g})$. Then we may
argue as in Theorem \textbf{2.1} to conclude that
$\frac{\hat{g}}{\hat{T}}$ is entire which will prove the theorem.
As in the proof of Theorem \textbf{2.1} we have
$Z_{\hat{T}}\subset Z_{\hat{g}}$. Without loss of generality we
can assume $x_{0}=0 .$ If $Ord_{x_0}(\hat{T}) = m_1$ and
$Ord_{x_0}( \hat{g}) = m_2$ then there exists holomorphic
functions $\varphi ,$  $\psi_1$ and $\psi_2$ such that $$
\hat{T}(z) = (z_n - \varphi(z^{'}))^{m_1} \psi_1(z) $$ and $$
\hat{g}(z) = (z_n - \varphi(z^{'}))^{m_2} \psi_2(z)$$ in a
neiborhood $V$ (in $\mathbb{C}^{n}$) of the origin, where $\psi_1$
and $\psi_2$ are zero free in $V.$ Here $z^{'} = (z_1, z_2, \cdots
z_{n-1}) \in \mathbb{C}^{n-1} .$

Since $\hat{g} / \hat{T} \in L^2,$ the above implies that,
$$\int_{[-a,a]^{n}}\frac{1}{|x_{n}-\varphi(x^{'})|^{2(m_{1}-m_{2})}}dx<\infty,$$
for some $a > 0.$  By a change of variable we get,
$$\int_{[-a,a]^{n-1}}\left(\int_{-a-\varphi(x^{'})}^{a-\varphi(x^{'})}\frac{1}{r^{2(m_{1}-m_{2})}}dr\right)dx^{'}<\infty.$$
Now, since $\varphi(0)=0$, if we choose $0<\varepsilon<a$, then
there exists $0<\delta< a$ such that
$|\phi(x)|<\varepsilon\vspace*{0.1in}~~\forall~
x^{'}\in[-\delta,\delta]^{n-1}$. Therefore,
$$\int_{[-\delta,\delta]^{n-1}}\left(\int_{-a+\varepsilon}^{a-\varepsilon}\frac{1}{r^{2(m_{1}-m_{2})}}dr\right)dx^{'}<\infty$$
implying that
$$\int_{-a+\varepsilon}^{a-\varepsilon}\frac{1}{r^{2(m_{1}-m_{2})}}dr<\infty.$$
Hence $m_{2} \geq m_{1}$, which finishes the proof.
\end{proof}

Next, suppose that $T$ is a radial distribution on
$\mathbb{R}^{n}$. Then $\hat{T}$ is a function of $(z_{1}^{2}+
z_{2}^{2} + \cdot\cdot +z_{n}^{2})^{\frac{1}{2}}$ and the
assignment $$\hat{T}(z_{1},z_{2},\cdot\cdot,z_{n})=G_{T}(s),$$
where $s^{2}=z_{1}^{2}+ z_{2}^{2} + \cdot\cdot +z_{n}^{2},$
defines an even entire function $G_{T}$ on the complex plane
$\mathbb{C}$ of exponential type and at most polynomial growth on
$\mathbb{R}$. The converse also holds. If the entire function
$G_{T}$ has only real zeros then $Z_{\hat{T}}$ (in
$\mathbb{C}^{n}$) is a disjoint union of sets of the form
$\{(z_{1},z_{2},\cdot\cdot,z_{n}):z_{1}^{2}+ z_{2}^{2} +
\cdot\cdot +z_{n}^{2}=a\}$ for $a>0$. It is easy to see that such
$T$ satisfies the condition $\textbf{(a)}$ of Theorem
\textbf{2.1}. Our next theorem shows that condition
$\textbf{(b)}$ of Theorem \textbf{2.1} is not necessary if we are
dealing with radial distributions of the above kind.

\begin{thm}
Let $T$ be a compactly supported radial distribution on
$\mathbb{R}^{n}$ such that the zeros of the entire function
$G_{T}(s)$ are contained in $\mathbb{R}-\{0\}$. If $f\in
L^{p}(\mathbb{R}^{n})$, $1\leq p\leq \frac{2n}{n-1}$ and $f*T$ is
compactly supported then $f$ is compactly supported.
\end{thm}

\begin{proof}
Let $f*T=g$ and let
$0<\lambda_{1}<\lambda_{2}<\lambda_{3}\cdot\cdot\cdot$ be the
positive zeros of $G_{T}(s)$ with multiplicities
$m_{1},m_{2},\cdot\cdot\cdot.$ We have $\hat{T}\hat{f}=\hat{g}.$
As in the previous case we will show that
$\frac{\hat{g}}{\hat{T}}$ is entire. It clearly suffices to show
that, $(z_{1}^{2}+ z_{2}^{2} + \cdot\cdot
+z_{n}^{2}-\lambda_{k}^{2})^{m_{k}}$ divides $\hat{g}$. Now
$\frac{G_{T}(s)}{s^{2}-\lambda_{k}^{2}}$ is an even entire
function of exponential type on $\mathbb{C}$ and is of at most
polynomial growth on $\mathbb{R}$. It follows that there exits a
compactly supported radial distribution $V$ on $\mathbb{R}^{n}$
such that
$$G_{V}(s)=\frac{G_{T}(s)}{s^{2}-\lambda_{k}^{2}}.$$ Now,
$$(z_{1}^{2}+ z_{2}^{2} + \cdot\cdot
+z_{n}^{2}-\lambda_{k}^{2})\frac{\hat{T}}{(z_{1}^{2}+ z_{2}^{2} +
\cdot\cdot +z_{n}^{2}-\lambda_{k}^{2})}\hat{f}=\hat{g}$$ implies
that \bea \label{eqn a}(-\Delta - \lambda_{k}^{2})(V*f)=g.\eea
Convolving $f$ with a radial $C_{c}^{\infty}$ function we may
assume that $V*f\in L^{p}(\mathbb{R}^{n})$, $1\leq p\leq
\frac{2n}{n-1}$. Note that $-\Delta - \lambda_{k}^{2}$ is a
distribution supported at the origin and satisfies the conditions
in Theorem \textbf{2.1}. It follows that $V*f$ is compactly
supported. Taking Fourier transform in (\ref{eqn a}) we obtain
that $(z_{1}^{2}+ z_{2}^{2} + \cdot\cdot
+z_{n}^{2}-\lambda_{k}^{2})$ divides $\hat{g}$. This surely can be
repeated to prove that $\frac{\hat{g}}{\hat{T}}$ is entire. The
proof now can be completed as in the previous case.
\end{proof}

In our next result we show that assuming $T$ is a compactly
supported positive distribution(i.e $T(\phi)\geq 0$ if $\phi\geq
0$) gives us precise information about the support of the
function $f$. Recall that a positive distribution is a positive
measure.

\begin{thm}
Let $T$ be a compactly supported radial positive measure with
supp $T=\overline{B_{R_{1}}}$. Assume that the entire function
$G_{T}(s)$ has only real zeros. If $f\in L^{p}(\mathbb{R}^{n})$,
$1\leq p\leq \frac{2n}{n-1}$ and $f*T=g$ with supp $g\subseteq
B_{R_{2}}$ then $f$ is compactly supported and supp $f\subseteq
B_{R_{2}-R_{1}}.$
\end{thm}

We start with the following lemma which is a simple application
of the Phragman-Lindeloff theorem.

\begin{lem}
Let $A(s)$ be an entire function of exponential type on
$\mathbb{C}$ and $0 < R_1 < R_2 < \infty. $ Suppose that
$|A(s)|\leq e^{R_{2}|s|}$ $\forall s\in \mathbb{C}$ and
\begin{itemize}
\item [\textbf{(a)}] $|A(is)|\leq e^{(R_{2}-R_{1})|s|}~~~\forall s\in\mathbb{R}.$
\item[\textbf{(b)}] $|A(s)|\leq e^{(R_{2}-R_{1})|s|}~~~\forall s\in\mathbb{R}.$
\end{itemize}
Then $|A(s)|\leq e^{(R_{2}-R_{1})|s|}~~~\forall s\in\mathbb{C}.$
\end{lem}

\begin{proof}
Define
$$H(s)=\frac{A(s)}{e^{(R_{2}-R_{1})(s)}},~~~s\in\mathbb{C}.$$ By
the given condition $H$ is an entire functon of exponential type
on $\mathbb{C}$. Also $H$ is bounded on real and imaginary axis.
Now consider the region $\Omega=\{s:\txt{Im}~s>0
~~\txt{and}~~\txt{Re}~s>0 \}$ which is a sector of angle
$\frac{\pi}{2}$. Then $H$ is bounded on $\partial\Omega$ and we
can find $P>0$ and $b<2$ such that $H(s)\leq
Pe^{|s|^{b}}~~~\forall z\in\Omega.$ By the Phragman-Lindeloff
theorem $H$ is bounded on $\Omega$. We can repeat the argument in
other quadrants. Hence the lemma follows.
\end{proof}

\begin{proof} [Proof of Theorem {\rm{\textbf{2.8}}}] Let $\mu$ be the compactly supported radial positive measure which defines the distribution $T$.
Then $f*\mu=g$. By Theorem \textbf{2.7} we already know that $f$ is compactly supported. In particular $f\in L^{1}(\mathbb{R}^{n})$.
Also $\hat{f}=\frac{\hat{g}}{\hat{\mu}}$ is an entire function of exponential type(by Malgrange's theorem).
Proof will be completed by Lemma \textbf{2.9} and the Paley-Wiener theorem once we prove that
$$|\hat{\mu}(iy)|\geq c_{\epsilon}e^{(R_{1}-\epsilon)|y|}~~\forall\epsilon>0,~~\forall y\in\mathbb{R}^{n}.$$
 Now, $$\hat{\mu}(iy)=\int_{|x|\leq R_{1}}e^{x\cdot y}d\mu(x).$$ Given $\epsilon> 0$, it is possible to choose
 a fixed radius $\delta>0$ such that $$x\cdot y\geq (R_{1}-\epsilon)|y|$$ for all $x$ in a $\delta-$nbd $B_{\delta}$ of $R_{1}\frac{y}{|y|}$.
 Hence \beas\hat{\mu}(iy)&\geq& \int_{x\in B_{\delta}}e^{x\cdot y}d\mu(x)\\&\geq & c(\delta)e^{(R_{1}-\epsilon)|y|}.\eeas
 Notice that we need supp $\mu= \overline{B_{R_{1}}}$ here. This finishes the proof.
\end{proof}

\begin{rem}
When $T=\chi_{ B_{r}}$ or $\mu_{r}$ this improves the result of
Sitaram in \cite{S}. Theorem \textbf{2.8} is also proved by
Volchkov in \cite{V} in a different way.
\end{rem}

The following theorem shows that the class of distributions which
satisfies the conditions in Theorem \textbf{2.7} is large. Notice
that if $G$ is an even entire function of exponential type on
$\mathbb{C}$ whose zeros are all nonzero reals and $T$ is a
radial, compactly supported distribution on $\mathbb{R}^{n}$
defined by $$\hat{T}(z_{1},z_{2},\cdot\cdot\cdot
z_{n})=G((z_{1}^{2}+z_{2}^{2}+\cdot\cdot\cdot
+z_{n}^{2})^{\frac{1}{2}})$$ then $T$ satisfies the conditions in
Theorem \textbf{2.7}

\begin{thm}
Let $\phi:\mathbb{R}\rightarrow\mathbb{R}$ be a positive even
$C^{2}$- function. Assume that $\phi$ is increasing on $[0,1].$
Then the entire function (on $\mathbb{C}$)
$$G(z):=\int_{-1}^{1}\phi(t)e^{-itz}dt$$ has only real zeros.
\end{thm}

Proof of the above requires several lemmas.

\begin{lem}
{\rm{\textbf{(I)}}} Let $g$ be a positive $C^{1}$ integrable
function on $[0,a)$ such that both $g$ and $g'$ are strictly
increasing on $[0,a)$. Then $$I=\int_{0}^{a}g(t)cos~t~dt$$ is non
zero if $a=2n\pi+\theta$ or $2n\pi+\pi+\theta$,
$0\leq\theta\leq\frac{\pi}{2}$.

\indent {\rm{\textbf{(II)}}} Let $g$ be as above with $g(0)=0$.
Then $$J=\int_{0}^{a}g(t)sin~t~dt$$ is non zero if
$a=2n\pi+\frac{\pi}{2}+\theta$ or $2n\pi+\frac{3\pi}{2}+\theta$,
$0\leq\theta\leq\frac{\pi}{2}$.
\end{lem}

\begin{proof}
\textbf{(I)} Case-1 : Let $a=2n\pi+\theta$,
$0\leq\theta\leq\frac{\pi}{2}$. Then \beas
I\geq\int_{0}^{2n\pi}g(t)cos~t~dt&=&\sum_{k=0}^{n-1}I_{k}\eeas
where \beas
I_{k}=\int_{2k\pi}^{2k\pi+2\pi}g(t)cos~t~dt&=&\int_{0}^{2\pi}g(2k\pi+t)cos~t~dt.\eeas
First, consider $I_{0}$.
$$I_{0}=\int_{0}^{\frac{\pi}{2}}G_{0}(t)cos~t~dt$$ where
$$G_{0}(t)=g(2\pi-t)-g(\pi+t)-g(\pi-t)+g(t).$$ Now,
$G_{0}(\frac{\pi}{2})=0$ and
$$G_{0}'(t)=-g'(2\pi-t)-g'(\pi+t)+g'(\pi-t)+g'(t)$$ is negative
by the assumption on $g$. It follows that $G_{0}(t)>0$ for $t\in
[0,\frac{\pi}{2})$. Hence $I_{0}>0$. Notice that each $I_k$
 is given by an integral $\int_0^{2\pi} G_k(t)~dt$ where $G_k$ is just $G_0$ translated
 by a multiple of $\pi.$ Hence each $I_{k}>0$ which implies that $I$ is non zero.

Case-2: Let $a=2n\pi+\pi+\theta$, $0\leq\theta\leq\frac{\pi}{2}$.
Then \beas -I\geq - \int_{0}^{2n\pi+\pi}g(t)cos~t~dt&=&
\bar{I}+\sum_{k=0}^{n-1}\bar{I_{k}}\eeas where $\bar{I}=-
\int_{0}^{\pi}g(t)cos~t~dt$ and \beas
\bar{I_{k}}=-\int_{(2k+1)\pi}^{(2k+1)\pi+2\pi}g(t)cos~t~dt&=&\int_{0}^{2\pi}g((2k+1)\pi+t)cos~t~dt.\eeas
Now $\bar{I}=
\int_{0}^{\frac{\pi}{2}}\left[g(\pi-t)-g(t)\right]cos~t~dt>0.$
Also as in the previous case $\bar{I_{k}}>0$. Therefore $I$ is
non zero.

\textbf{(II)} Case-1: Let $a=2n\pi+\frac{\pi}{2}\theta$,
$0\leq\theta\leq\frac{\pi}{2}$.  Then \beas
J\geq\int_{0}^{2n\pi+\frac{\pi}{2}}g(t)sin~t~dt&=&\sum_{k=0}^{n-1}J_{k}\eeas
where \beas
J_{k}=\int_{2k\pi\frac{\pi}{2}}^{2k\pi+\frac{\pi}{2}+2\pi}g(t)sin~t~dt&=&\int_{\frac{\pi}{2}}^{\frac{\pi}{2}+2\pi}g(2k\pi+t)sin~t~dt.\eeas
First consider $J_{0}$.
$$J_{0}=\int_{0}^{\frac{\pi}{2}}E_{0}(t)sin~t~dt$$ where
$$E_{0}(t)=g(2\pi+t)-g(2\pi-t)-g(\pi+t)+g(\pi-t).$$ Now,
$E_{0}(0)=0$ and
$$E_{0}'(t)=g'(2\pi+t)+g'(2\pi-t)-g'(\pi+t)-g'(\pi-t)$$ is
positive by assumption on $g$. It follows that $E_{0}(t)>0$ for
$t\in (0,\frac{\pi}{2}]$. Hence $J_{0}>0$. Similarly each
$J_{k}>0$ which implies that $J$ is non zero.

Case-2: Let $a=2n\pi+\frac{3\pi}{2}+\theta$,
$0\leq\theta\leq\frac{\pi}{2}$.Then \beas -J\geq -
\int_{0}^{2n\pi+\frac{3\pi}{2}}g(t)sin~t~dt&=&
\bar{J}+\sum_{k=0}^{n-1}\bar{J_{k}}\eeas where $\bar{J}=-
\int_{0}^{\frac{3\pi}{2}}g(t)sin~t~dt$ and \beas
\bar{J_{k}}=-\int_{2k\pi+\frac{3\pi}{2}}^{2k\pi+\frac{3\pi}{2}+2\pi}g(t)sin~t~dt&=&\int_{\frac{\pi}{2}}^{\frac{\pi}{2}+2\pi}g((2k+1)\pi+t)sin~t~dt.\eeas
$\bar{J} = \int_{0}^{\frac{\pi}{2}}E(t)sin~t~dt$ where
$$E(t)=g(\pi+t)-g(\pi-t)-g(t).$$ Now $E(0)=0$ and
$$E'(t)=g'(\pi+t)+g'(\pi-t)-g'(t)$$ is positive by assumptions on
$g$. It follows that $E(t)>0$ for $t\in (0,\frac{\pi}{2}].$ Hence
$\bar{J}>0$. Also as in the previous case $\bar{J_{k}}>0$.
Therefore $J$ is non zero.
\end{proof}

\begin{lem}
{\rm{\textbf{(I)}}} Let $g$ be a non negative  continious
integrable strictly increasing function on $[0,a)$. Then,
$$I:=\int_{0}^{a}g(t)cos t$$ is non zero if
$a=\frac{\pi}{2}+k\pi$ for some non negative integer $k$.

\indent {\rm{\textbf{(II)}}} Let $g$ be as above. Then,
$$J:=\int_{0}^{a}g(t)sin t$$ is non zero if $a=k\pi$ for some
positive integer $k$.
\end{lem}

\begin{proof}
Let $a=\frac{\pi}{2}+k\pi$ for some non negative integer $k$.
Then,
$$I=\int_{0}^{\frac{\pi}{2}}g(t)cost+\sum_{j = 0}^{(k-1)}I_{j}$$
where
$$I_{j}=\int_{\frac{\pi}{2}+j\pi}^{\frac{\pi}{2}+(j+1)\pi}g(t)cost~dt.$$
If $k$ is even we can write
$$I=\int_{0}^{\frac{\pi}{2}}g(t)cost+\sum_{j=0}^{\frac{k-2}{2}}(I_{2j}+I_{2j+1}).$$
By a change of variable we get
$$I_{0}+I_{1}=\int_{0}^{\pi}\left[g\left(\pi+\frac{\pi}{2}+t\right)-g\left(\frac{\pi}{2}+t\right)\right]sintdt$$
which is positive since $g$ is strictly increasing. Similarly each
$I_{2j}+I_{2j+1}$ is positive. Hence $I$ is positive. If $k$ is
odd then we can write
$$I=\int_{0}^{\frac{3\pi}{2}}g(t)costdt+\sum_{j=1}^{\frac{k-1}{2}}(I_{2j-1}+I_{2j}).$$
Again using a change of variable we get
$$I_{1}+I_{2}=\int_{0}^{\pi}\left[g\left(\frac{\pi}{2}+\pi+t\right)-g\left(\frac{\pi}{2}+2\pi+t\right)\right]sintdt$$
which is negative since $g$ is strictly increasing. Similarly each
$I_{2j-1}+I_{2j}$ is negative. Also \beas
\int_{0}^{\frac{3\pi}{2}}g(t)costdt
&<&\int_{0}^{\frac{\pi}{2}}g(t)costdt+\int_{\pi}^{\pi+\frac{\pi}{2}}g(t)costdt\\&<&\int_{0}^{\frac{\pi}{2}}[g(t)-g(\pi+t)]costdt\eeas
is negative. Therefore $I$ is negative. Hence \textbf{(I)} is
proved. \textbf{(II)} can be proved using similar type of
arguments.
\end{proof}

\begin{lem}
{\rm{\textbf{(I)}}} Let $g$ be a non negative increasing
integrable $C^{2}$function on $[0,1)$ such that for some $M>1$,
$Mg(t)+g''(t)\geq 0$  $\forall t\in [0,1)$. Then, for each fixed
$y>M$ the function
$$F_{y}(x):=\int_{0}^{1}g(t)(e^{yt}+e^{-yt})cos~xtdt$$ can vanish
atmost once  in each of the interval
$[\frac{\pi}{2}+k\pi,\frac{\pi}{2}+(k+1)\pi]$, where $k$ is a non
negative integer.

\indent {\rm{\textbf{(II)}}} Let $g$ be as above. Then, for each
fixed $y>M$ the function
$$G_{y}(x)=\int_{0}^{1}g(t)(e^{yt}+e^{-yt})sin~xtdt$$ can vanish
atmost once  in each of the interval $[k\pi,(k+1)\pi]$, where $k$
is a non negative integer.
\end{lem}

\begin{proof}
To prove \textbf{(I)} first note that we can write $F_{y}(x)$ and
$F'_{y}(x)$ in the following way :
$$F_{y}(x)=\frac{1}{x}\int_{0}^{x}g\left(\frac{t}{x}\right)\left(e^{t\frac{y}{x}}+e^{-t\frac{y}{x}}\right)cos
tdt$$ and
$$F'_{y}(x)=-\frac{1}{x}\int_{0}^{x}\frac{t}{x}g\left(\frac{t}{x}\right)\left(e^{t\frac{y}{x}}+e^{-t\frac{y}{x}}\right)sin
tdt.$$ Now, if possible assume that there exists $y_{0}>M$ and a
non negative integer $k_{0}$ such that the interval
$[\frac{\pi}{2}+k_{0}\pi,\frac{\pi}{2}+(k_{0}+1)\pi]$ contains at
least two zeros of the function $F_{y_{0}}(x)$. Beasause of the
given conditions an easy calculation shows that the functions
$g(\frac{t}{x})(e^{t\frac{y_{0}}{x}}+e^{t\frac{y_{0}}{x}})$ and
$\frac{t}{x}g(\frac{t}{x})(e^{t\frac{y_{0}}{x}}+e^{t\frac{y_{0}}{x}})$
on the interval $[0,x)$ satisfy the conditions of \textbf{(I)} and
\textbf{(II)} of Lemma \textbf{(2.11)} respectively. Hence,
$F_{y_{0}}(x)$ and $F'_{y_{0}}(x)$ can not vanish in the intervals
$[\frac{\pi}{2}+k\pi+\frac{\pi}{2},\frac{\pi}{2}+(k+1)\pi]$ and
$[\frac{\pi}{2}+k\pi,\frac{\pi}{2}+k\pi+\frac{\pi}{2}]$
respectively. Therefore, $F_{y_{0}}(x)$ vanishes at least twice in
the interval
$[\frac{\pi}{2}+k\pi,\frac{\pi}{2}+k\pi+\frac{\pi}{2}]$ which
implies, by  Rolles theorem that $F'_{y_{0}}(x)$ has at least one
zero in the same interval, which is a contradiction. This finishes
the proof of \textbf{(I)}. Using similar type of arguments we can
prove \textbf{(II)} also.
\end{proof}

\begin{lem}
Let $g$ be an even or odd continious integrable function on
$(-1,1)$ such that on $[0, 1]$ it is non negative, increasing and
$C^{2}$. Assume that for some $M>1$, $Mg(t)+g''(t)\geq 0$ $\forall
t\in [0,1)$. Let the entire function
$$H_{1}(z):=\int_{-1}^{1}g(t)e^{-izt}dt$$ has a non real zero.
Then the entire function
$$H_{2}(z):=\int_{-1}^{1}tg(t)e^{-izt}dt$$also has a non real
zero.
\end{lem}

\begin{proof}
First assume that $g$ is even. Since $g$ is also real valued,
there exists $x_{0}>0$ and $y_{0}>0$ such that $H_{1}$ is zero at
$z_{0}=x_{0}+iy_{0}$. Now, if possible assume that $H_{2}$ has
only real zeros, i.e for any $z=x+iy$, $y\neq 0$,
$$Re~H_{2}(z)=\int_{0}^{1}t\phi(t)(e^{yt}-e^{-yt})cosxt~dt,$$ and
$$Im~H_{2}(z)=-\int_{0}^{1}t\phi(t)(e^{yt}+e^{-yt})sinxt~dt$$ can
not vanish simultaneously. But this implies that, if we define
the smooth function $F:\mathbb{R}^{2}\rightarrow\mathbb{R}$ by
$$F(x,y)= Re~H_{1}(x+iy)=\int_{0}^{1}g(t)(e^{yt}+e^{-yt})cosxtdt$$
then the gradient vector $$\nabla
F(x,y)=\left(-\int_{0}^{1}tg(t)(e^{yt}+e^{-yt})sinxtdt,\int_{0}^{1}tg(t)(e^{yt}-e^{-yt})cosxtdt\right)\neq
0$$ whenever $z=x+iy$ is not real i.e $y\neq 0$. Therefore, the
zero set of $F$ in the open upper half plane defines a smooth
$1$-dimensional manifold.

By \textbf{(I)} of Lemma \textbf{(2.12)}, the connected component
of the zero set through $(x_{0},y_{0})$ in the open upper half
plane (call it $C$) is contained in the region
$R:=\{(x,y):\frac{\pi}{2}+k\pi<x<\frac{\pi}{2}+(k+1)\pi,y>0\}$ for
some non negative integer $k$. Now, it is clear that either $C$
will cross the $x-$axis or it will be entirely above the $x-$axis
in which case the closure $\bar{C}$ will include points on the
$x-$axis (this follows from Lemma \textbf{(2.14)} as $C$ cannot be
a closed curve or go "upwards" like a parabola). Hence we may
parametrize the curve (or a portion of it) by $\gamma : [0, 1]
\rightarrow \bar{R} $ such that $\gamma(0) = (x_0, y_0)$ and
$\gamma(1) = (u_o, 0)$ ( the point at which $\bar{C}$ hits the
$x-$axis). Notice that $\frac{\pi}{2} + k\pi < u_0 < \frac{\pi}{2}
+ (k+1)\pi$ and $\gamma$ is smooth with $\gamma^{'}(s) \neq 0$ for
$s \in (0, 1).$ Now, identifying $\mathbb{R}^{2}$ with
$\mathbb{C}$ consider the function $H_{1}\circ\gamma$. It is easy
to see that, this is purely imaginary valued continious function
on $[0,1]$, smooth on $(0,1)$, which vanishes at $0$ and $1$.
Since $\gamma'$ is non zero on $(0,1)$, applying Rolle's theorem
to the function $i(H_{1}\circ\gamma)$ we get that
$\int_{-1}^{1}t\phi(t)e^{-i\gamma(s_{0})t}dt=0$ for some
$s_{0}\in(0,1)$, which is a contradiction, because $\gamma(s_{0})$
is not real.
This finishes the proof when $g$ is even. When $g$ is odd the
proof is almost similar except the fact that instead of finding a
path (C) on which $H_{1}$ is purely imaginary ($0$ included) we
find a path on which $H_{1}$ is real.
\end{proof}

\begin{proof}[Proof of Theorem {\rm{\textbf{2.11}}}] If possible assume that $G$ has
a non real zero. Now, from the given conditions it is easy to see
that for some large $M>0$ $M\phi(t)+\phi''(t)\geq 0$ and hence for
any positive integer $n$ $M(t^{n}\phi(t))+(t^{n}\phi(t))''\geq 0$,
for all $t\in [0,1]$. By Lemma \textbf{2.15} and using induction
we can say that for each positive integer $n$ the entire function
$$G_{n}(s):=\int_{-1}^{1}\phi_{n}(t)e^{-its}dt$$ has a non real
zero, where $$\phi_{n}(t):=t^{n}\phi(t) ~~\forall
t\in\mathbb{R}.$$ Since
$$\phi_{n}'(t)=nt^{(n-1)}\phi(t)+t^{n}\phi'(t)$$ and
\beas\phi_{n}''(t)&=& n(n-1)t^{(n-2)}\phi(t)+4nt^{n-1}\phi'(t)
+t^{n}\phi''(t)\\&=&
t^{(n-2)}[n(n-1)\phi(t)+t^{2}\phi''(t)]+4nt^{(n-1)}\phi'(t),\eeas
by the given conditions it follows that, for some large positive
integer
 $N$ (we can take $N$ to be even) $\phi'_{N}(t)\geq 0$ and $\phi''_{N}(t)\geq 0$
  for all $t\in[0,1]$, i.e $\phi_{N}$ and $\phi_{N}'$ both are increasing on [0,1].
  Now, since $\phi_{N}$ is even and real valued, we will get a contradiction if we
   can prove that $G_{N}(s)$ has no zero in $\{s\in\mathbb{C}:s=x+iy,~~x>0,y>0\}$.
    Now, \beas G_{N}(s) &=& 2\int_{0}^{1}\phi_{N}(t)\left(e^{-its}+e^{its}\right)dt \\
    &=& \int_{0}^{1}\phi_{N}(t)\left(e^{-itx}e^{ty}+e^{itx}e^{-ty}\right)dt \\
    &=&\frac{2}{x}\int_{0}^{x}\phi_{N}\left(\frac{t}{x}\right)\left(e^{-it}e^{t\frac{y}{x}}+e^{it}e^{-t\frac{y}{x}}\right)dt.\eeas
    Therefore, $$Re~G_{N}(s)=\frac{2}{x}\int_{0}^{x}\phi_{N}\left(\frac{t}{x}\right)\left(e^{t\frac{y}{x}}+e^{-t\frac{y}{x}}\right)cos~t~dt$$
    and $$-Im~G_{N}(s)=\frac{2}{x}\int_{0}^{x}\phi_{N}\left(\frac{t}{x}\right)\left(e^{t\frac{y}{x}}-e^{-t\frac{y}{x}}\right)sin~t~dt.$$
    Since $\phi_{N}$ and $\phi_{N}'$ both are increasing on [0,1], it is easy to see that the functions
    $\phi_{N}\left(\frac{t}{x}\right)\left(e^{t\frac{y}{x}}+e^{-t\frac{y}{x}}\right)$ and
    $\phi_{N}\left(\frac{t}{x}\right)\left(e^{t\frac{y}{x}}-e^{-t\frac{y}{x}}\right)$ on
    the interval $[0,x)$ satisfy the assumptions in Lemma \textbf{2.12}. Therefore,
    both Re $G_{N}(s)$ and Im $G_{N}(s)$ can not be simultaneously zero in the first
     quadrant which finishes the proof.
\end{proof}

\section{Support theorems on non compact symmetric spaces}

In this section we prove support theorems on non compact symmetric
spaces. Let $G$ be a connected, non compact semisimple Lie group
with finite center. Let $K\subseteq G$ be a fixed maximal compact
subgroup and $X=G/K$, the associated Riemannian space of non
compact type. Endow $X$ with the $G$-invariant Riemannian
structure induced from the Killing form. Let $dx$ denote the
Riemannian volume element on $X$. We study convolution equations
of the form $f*T=g$, where $f\in C^{\infty}(X)\cap L^{p}(X)$, $T$
a $K$-biinvariant compactly supported distribution on $X$ and
$g\in C_{c}^{\infty}(X)$. We show that under natural assumptions
on the zero set of the spherical Fourier transform of $T$, $f$
turns out to be compactly supported. (The function $f$ is assumed
to be smooth only to make sure that the convolution $f*T$ is well
defined). Before we state our results we recall necessary details.
For any unexplained notation see \cite{He}.

Let $G=KAN$ be an Iwasawa decomposition of $G$ and $\textbf{a}$
be the Lie Algebra of $A$. Let $\textbf{a}^{*}$ be the real dual
of $\textbf{a}$ and $\textbf{a}_{\mathbb{C}}^{*}$ its
complexification. Then for any $g\in G$, $g=k(g)exp H(g)n(g)$
where $k(g)\in K$, $H(g)\in \textbf{a}$, $n(g)\in N$. Let $M$ be
the centralizer of $A$ in $K$. For a suitable function $f$ on
$X$, the Helgason Fourier transform is defined by
$$\tilde{f}(\lambda,
k)=\int_{G}f(x)e^{(i\lambda-\rho)H(x^{-1}k)}dk,$$ where $\rho$ is
the half sum of positive roots and $\lambda\in \textbf{a}^{*}$.
We note that $\tilde{f}(\lambda,k)=\tilde{f}(\lambda,kM)$ and so
sometimes we will write $\tilde{f}(\lambda,b)$ where $b=kM$.

For each $\lambda\in \textbf{a}_{\mathbb{C}}^{*}$, let
$\phi_{\lambda}$ be the elementary spherical function given by :
$$\phi_{\lambda}(g)=\int_{K}e^{(i\lambda-\rho)H(x^{-1}k)}~dk.$$They
are the matrix elements of the spherical principal
representations $\pi_{\lambda}$ of $G$ defined for $\lambda\in
\textbf{a}_{\mathbb{C}}^{*}$ on $L^{2}(K/M)$ by
$$(\pi_{\lambda}(x)v)(b)=e^{(i\lambda-\rho)H(x^{-1}b)}v(k(x^{-1}b)),$$
where $v\in L^{2}(K/M)$. The representations $\pi_{\lambda}$ is
unitary if and only if $\lambda\in \textbf{a}^{*}$. They are also
irreducible if $\lambda\in \textbf{a}^{*}$. For $f\in L^{1}(X)$,
the group Fourier transform $\pi_{\lambda}(f)$, defined by
$$\pi_{\lambda}(f)=\int_{X}f(x)\pi_{\lambda}(x)dx$$ is a bounded
linear operator on $L^{2}(K/M)$. Its action is given by
$$(\pi_{\lambda}(f)v)(b)=\left(\int_{K/M}v(k)dk\right)\tilde{f}(\lambda,b).$$
We also have the Plancherel formula which says that $f\rightarrow
\tilde{f}(\lambda,b)$ is an isometry from $L^{2}(X)$ onto
$L^{2}(\textbf{a}^{*}\times K/M,~~|c(\lambda)|^{-2}d\lambda)$
where $c(\lambda)$ is the Harish-Chandra $c$-function. In
particular,
$$\int_{X}|f(x)|^{2}dx=|W|^{-1}\int_{\textbf{a}^{*}}\int_{K/M}|\tilde{f}(\lambda,w)|^{2}|c(\lambda)|^{-2}d\lambda
dk.$$ Next we comment on the pointwise existence of the Helgason
Fourier transform. For $1\leq p\leq 2$, define
$S_{p}=\textbf{a}^{*}+iC_{\rho}^{p}$, where $C_{\rho}^{p}$ is the
convex hull of $\{s(\frac{2}{p}-1)\rho:s\in W\}$, $W$ being the
Weyl group. let $S_{p}^{^{0}}$ be the interior of $S_p.$. The
following result from \cite{PSSS} proves the existence of Helgason
Fourier transform pointwise.

\begin{thm}
Let $f\in L^{p}(X)$, $1\leq p\leq 2$. Then $\exists$ a subset
$B(f)\subseteq K$, of full measure such that
$\tilde{f}(\lambda,b)$ exists $\forall~~ b\in B$ and $\lambda\in
S_{p}^{0}$. Moreover, for every $b\in B(f),$ fixed,
$\lambda\rightarrow \tilde{f}(\lambda,b)$ is holomorphic on
$S_{p}^{0}$ and
$||\tilde{f}(\lambda,\cdot)||_{L^{1}(K)}\rightarrow 0$ as
$|\lambda|\rightarrow\infty$ in $S_{p}^{0}$.
\end{thm}

\begin{rem}
\textbf{(1)} When $p=1$ we have
$||\tilde{f}(\lambda,\cdot)||_{L^{1}(K)}\leq ||f||_{1}$
$\forall\lambda\in S_{1}$.

\indent \textbf{(2)} When $p=2$, existence of
$\tilde{f}(\lambda,b)$ is provided by the Plancherel theorem.
\end{rem}

We also have the Paley-Wiener theorem for compactly supported
functions and distributions.

\begin{thm}
The Fourier transform is a bijection from $C_{c}^{\infty}(X)$ to
$C^{\infty}$ functions $\psi$ on $\bf{a}_{\mathbb{C}}^{*}\times
K/M$ satisfying
\begin{itemize}
\item[\textbf{(a)}] $\psi(\lambda,b)$ is holomorphic as a function
of $\lambda$.
\item[\textbf{(b)}] There is a constant $R\geq 0$
such that $\forall N>0$ $$Sup_{\lambda\in
\bf{a}_{\mathbb{C}}^{*},~ b\in K/M
}e^{-R|Im\lambda|}(1+|\lambda|)^{N}|\psi(\lambda,b)|<\infty.$$
\item[\textbf{(c)}]For any $\sigma$ in the Weyl group and $g\in G$
$$\int_{K/M}e^{-(i\sigma\lambda+\rho)H(g^{-1}k)}\psi(\sigma\lambda,kM)dk=\int_{K/M}e^{-(i\lambda+\rho)H(g^{-1}k)}\psi(\lambda,kM)dk.$$
\end{itemize}
\end{thm}
\noindent The above theorem extends to the case of distributions
too. See \cite{EHO}.

We restate these results as in \cite{SS}. Let $v_{j}$,
$j=0,1,2,...$ be an orthonormal basis for $L^{2}(K/M)$ where each
$v_{j}$ transform according to some irreducible unitary
representation of $K$ and $v_{0}$ is the constant function $1$ on
$K/M$. (Note that, for any $\lambda\in \textbf{a}^{*}$
$\pi_{\lambda}(k)v_{0}=v_{0}$ and $v_{0}$ is the essentially
unique vector with this property). Let $\widehat{K_{M}}$ consists
of all unitary irreducible representations of $K$ which have an
$M$ fixed vector. For $\delta\in \widehat{K_{M}}$ let
$\chi_{\delta}$ be its character. If $f\in C^{\infty}(X)$, then
$$f=\sum_{\delta\in \widehat{K_{M}}}\chi_{\delta}* f,$$ where the
convergence is in the $C^{\infty}(X)$ topology. It follows that
$f$ is compactly supported if and only if $\chi_{\delta}* f$ is
compactly supported for all $\delta$. We now state the
Paley-Wiener theorem in the following form:

\begin{thm}
Let $f\in L^{p}(X)$, $1\leq p\leq 2$ and $f=\chi_{\delta}* f$ for
some $\delta\in \widehat{K_{M}}$. Then $\tilde{f}(\lambda, b
)=a_{1}(\lambda)v_{i_{1}}(b)+a_{2}(\lambda)v_{i_{2}}(b)+\cdot\cdot\cdot+a_{n}(\lambda)v_{i_{n}}(b).$

\indent \textbf{(a)} If supp $f\subseteq B_{R}$, then each
$a_{i}(\lambda)$ extends to an entire function on
$\bf{a}_{\mathbb{C}}^{*}$ of exponential type $R$.

\indent \textbf{(b)} conversely if each $a_{i}$ extend to an
entire function of exponential type $R$ then supp $f\subseteq
B_{R}$.
\end{thm}

\begin{rem}
In \cite{SS} the above theorem is stated only for $f\in L^{1}(X)$.
But, this clearly extends to $f\in L^{p}(X),$ $1\leq p\leq 2$.
\end{rem}

We also recall that if $f$ is  $K-$biinvariant then the Helgason
Fourier transform is independent of $b$ and it reduces to the
spherical Fourier transform of $f$ defined by $$
\tilde{f}(\lambda) = \int f(x)~\varphi_{\lambda}(x)~dx .$$ If $T$
is a $K-$biinvariant compactly supported distribution then
$\tilde{T} (\lambda)$ can be defined similarly. We finish the
preliminaries with the following proposition.

\begin{prop}
Let $f\in L^{p}(X)\cap C^{\infty}(X)$, $1\leq p\leq 2$ and $T$ be
a compactly supported $K$-biinvariant distribution such that
$f*T$ is compactly supported. Then
$$(f*T)^{\tilde{}}(\lambda,b)=\tilde{f}(\lambda,b)\tilde{T}(\lambda).$$
\end{prop}

\begin{proof}
Since $L^{p}\subseteq L^{1}+L^{2}$, it suffices to prove this for
$L^{1}$ and $L^{2}$. If $\phi\in C_{c}^{\infty}(K
\backslash{G}/K)$ then $T*\phi=\phi*T\in
C_{c}^{\infty}(K\backslash{G}/K)$ and
$$(T*\phi)^{\tilde{}}(\lambda,b)=\tilde{T}(\lambda)\tilde{\phi}(\lambda).$$
Also if $f\in L^{1}$ or $L^{2}$ and $g\in
C_{c}^{\infty}(K\backslash G/K)$ then
$$(f*g)^{\tilde{}}(\lambda,b)=\tilde{f}(\lambda,b)\tilde{g}(\lambda).$$
Now, by assumption $f*T\in C_{c}^{\infty}(X).$ So
$$\left((f*T)*\phi\right)^{\tilde{}}(\lambda,b)=(f*T)^{\tilde{}}(\lambda,b)\tilde{\phi}(\lambda).$$
But $(f*T)*\phi=f*(T*\phi)$ and
$$\left(f*(T*\phi)\right)^{\tilde{}}(\lambda,b)=\tilde{f}(\lambda,b)\tilde{T}(\lambda)\tilde{\phi}(\lambda)$$
which proves the proposition.
\end{proof}

Now we are in a position to state the analogue of Theorem
\textbf{2.1} in the previous section. We first deal with the case
$1\leq p<2.$

\begin{thm}
Let $f\in L^{p}(X)\cap C^{\infty}(X)$, $1\leq p<2$ and $T$ be a
compactly supported $K$-biinvariant distribution. Assume that
$f*T$ is compactly supported. If all irreducible components of
$Z_{\tilde{T}}$ intersects $S_{p}^{0}$, then $f$ is compactly
supported.
\end{thm}

\begin{proof}
Let $f*T=g$, for $g\in C_{c}^{\infty}(X)$. We may assume that
$f=\chi_{\delta}*f$ and so $g=\chi_{\delta}*g$ as $T$ is
$K-$biinvariant. We have
$$\tilde{g}(\lambda,b)=a_{1}(\lambda)v_{i_{1}}+a_{2}(\lambda)v_{i_{2}}+\cdot\cdot\cdot+a_{n}(\lambda)v_{i_{n}},$$
where each $a_{i}(\lambda)$ extends to an entire function on
$\textbf{a}_{\mathbb{C}}^{*}$ of exponential type $R$ (for some
$R>0$), whose restriction to $\textbf{a}^{*}$ is bounded. Next, by
Proposition \textbf{3.6}
$$ (f*T)^{\tilde{}}(\lambda,b)=\tilde{f}(\lambda,b)\tilde{T}(\lambda).$$
It follows that
$$\tilde{f}(\lambda,b)=b_{1}(\lambda)v_{i_{1}}+b_{2}(\lambda)v_{i_{2}}+\cdot\cdot\cdot+b_{n}(\lambda)v_{i_{n}},$$
where $$a_{j}(\lambda)=\tilde{T}(\lambda)b_{j}(\lambda).$$ Now,
$b_{j}(\lambda)$ are holomorphic functions on $S_{p}^{0}$ and all
the irreducible components of $Z_{\tilde{T}}$ intersect
$S_{p}^{0}$. It immediately follows that
$\frac{a_{j}}{\tilde{T}}$ is an entire function of exponential
type. This finishes the proof.
\end{proof}

To prove the $L^{2}$ case we need to recall details about the
$\delta$-spherical transform and analyze the $c$- function in
detail. If $f\in C^{\infty}(X)$ then we have
$$f=\sum_{\delta\in\widehat{K_{M}}}d(\delta)\chi_{\delta}*f,$$
where $\widehat{K_{M}}$ consists of all unitary irreducible
representations of $K$ which have $M$-fixed vector. We also have
$L^{2}(K/M)=\bigoplus_{\delta\in\widehat{K_{M}}}V_{\delta}$, where
$V_{\delta}$ consists of the vectors in $L^{2}(K/M)$ that
transform according to the representation $\delta$ under the
$K$-action. Let $V_{\delta}^{M}=\{v\in
V_{\delta}:\delta(m)v=v~~\forall m\in M \}$. For
$\delta\in\widehat{K_{M}}$ define spherical functions of type
$\delta$ by $$\Phi_{\lambda,\delta}
(x)=\int_{K}e^{-(i\lambda+\rho)(H(x^{-1}k))}\delta(k)dk,~~~\lambda\in
\textbf{a}_{\mathbb{C}}^{*},~~x\in X. $$ Then,
$$\Phi_{\lambda,\delta}(k,x)=\delta(k)\Phi_{\lambda,\delta}(x),$$and
$$\Phi_{\lambda,\delta}(x)\delta(m)=\Phi_{\lambda,\delta}(x)~~~m\in
M.$$ If $f=d(\delta)\chi_{\delta}*f$, define its
$\delta$-spherical Fourier transform by
$$\tilde{f}(\lambda)=d(\delta)\int_{X}f(x)\Phi_{\lambda,\delta}^{*}(x)dx,$$
where * denotes the adjoint. If $\delta$ is the trivial
representation then $f\rightarrow\tilde{f}$ is the spherical
Fourier transform. In general
$\delta(m)\tilde{f}(\lambda)=\tilde{f}(\lambda)$ and so
$\tilde{f}(\lambda)\in Hom(V_{\delta},V_{\delta}^{M})$. If
$\tilde{f}(\lambda,kM)$ is the Helgason Fourier transform of $f$
then we have
$$\tilde{f}(\lambda)=d(\delta)\int_{K}\tilde{f}(\lambda,kM)\delta(k^{-1})dk,~~~\tilde{f}(\lambda,kM)=Trace(\delta(k)\tilde{f}(\lambda)).$$
The $\delta$-spherical Fourier transform is inverted by
$$f(x)=\frac{1}{|w|}Trace\left(\int_{\textbf{a}^{*}}\Phi_{\lambda,\delta}(x)\tilde{f}(\lambda)|c(\lambda)|^{-2}d\lambda\right).$$
For each $\delta\in \widehat{K_{M}}$, we also have the
$Q_{\delta}(\lambda)$ matrices which are $l(\delta)\times
l(\delta)$ matrices whose entries are polynomial factors in
$\lambda$. Here $l(\delta)=dim V_{\delta}^{M}$. The Paley-Wiener
theorem for $\delta$-spherical transform says the following: Let
$H^{\delta}(\textbf{a}^{*})$ stand for all the functions
$F:\textbf{a}_{\mathbb{C}}^{*}\rightarrow
Hom(V_{\delta},V_{\delta}^{M})$ such that
\begin{itemize}
\item[\textbf{(i)}] $F$ is holomorphic and is of exponential type.
\item[\textbf{(ii)}] $Q_{\delta}^{-1}F$ is holomorphic and Weyl
group invariant.
\end{itemize}

\begin{thm}
The $\delta$-spherical transform $f\rightarrow\tilde{f}$ is a
homeomorphism from $\{f\in
C_{c}^{\infty}(X):f=d(\delta)\chi_{\delta}*f\}$ onto
$H^{\delta}(\textbf{a}^{*})$.
\end{thm}
( See \cite{He} ).

We are now in a position to state the $L^{2}$ version of Theorem
\textbf{3.7} Also recall that if $G$ is a real rank one group then
$\textbf{a}$ and $ \textbf{a}^{*}$ may be identified with
$\mathbb{R}$ and $\textbf{a}_{\mathbb{C}}^{*}$ with $\mathbb{C}$.

\begin{thm}
{\rm{\textbf{(1)}}} Let $G$ be a real rank one group and $T$ be a
compactly supported $K$-biinvariant distribution such that all the
zeros of $\tilde{T}(\lambda)$ are real. If $f\in L^{2}\cap
C^{\infty}(G/K)$ and $f*T$ is compactly supported then $f$ is
compactly supported.

\indent {\rm{\textbf{(2)}}} Let $G$ be such that it has only one
congugacy class of Cartan subgroups. Let $T$ be a $K$-biinvariant
compactly supported distribution such that any irreducible
component of $Z_{\tilde{T}}$ intersected with $\bf{a}^{*}$ has
real dimension $(n-1)$. If $f\in L^{2}\cap C^{\infty}(X)$ and
$f*T$ is compactly supported then $f$ is compactly supported.
\end{thm}

\begin{proof}
\textbf{(1)}: In the rank one case it is known that
$\lambda\rightarrow c(\lambda)$ is a meromorphic function on
$\mathbb{C}$ with simple poles, all lying on the imaginary axis.
In particular $\lambda=0$ is a simple pole. It follows that
$|c(\lambda)|^{-2}=c(\lambda)c(-\lambda)$ is a holomorphic
function in a small strip containing the real line and the only
zero of $|c(\lambda)|^{-2}$ in that strip is $\lambda=0$, of
order $2$. As in the previous theorem we assume that
$f=d(\delta)\chi_{\delta}*f$ and so $g=d(\delta)\chi_{\delta}*g$.
Applying the $\delta$-spherical transform to $f*T=g$ we obtain
$$\tilde{T}(\lambda)\tilde{f}(\lambda)=\tilde{g}(\lambda).$$
Since $l(\delta)=dim V_{\delta}^{M}=1,$ both $\tilde{f}(\lambda)$
and $\tilde{g}(\lambda)$ are $1\times d(\delta)$ vectors. So, to
be consistent with previous notation we write
$$\tilde{g}(\lambda)=(a_{1}(\lambda),a_{2}(\lambda),\cdot\cdot\cdot
, a_{d(\delta)}(\lambda)),$$ and
$$\tilde{f}(\lambda)=(b_{1}(\lambda),b_{2}(\lambda),\cdot\cdot\cdot
, b_{d(\delta)}(\lambda)),$$ where
$$b_{j}(\lambda)=\frac{a_{j}(\lambda)}{\tilde{T}(\lambda)}.$$ By
the Paley-Wiener theorem (Theorem \textbf{3.4})
$\lambda\rightarrow a_{j}(\lambda)$ is an entire function of
exponential type and $\frac{a_{j}(\lambda)}{Q_{j}(\lambda)}$ is an
even entire function on $\mathbb{C}$. We also have
\bea\label{eqn11}\int_{\textbf{a}^{*}}\left|\frac{a_{j}(\lambda)}{\tilde{T}(\lambda)}\right|^{2}|c(\lambda)|^{-2}d\lambda<\infty.\eea
Now if $0\neq \lambda_{0}$ is a zero of $\tilde{T}(\lambda)$ of
order $k$, since $|c(\lambda_{0})|^{-2}\neq 0$ it readily follows
from (3.1) that $\lambda_{0}$ is a zero of $a_{j}(\lambda)$ of
order at least $k$. Next, suppose that $\lambda=0$ is a zero
$\tilde{T}(\lambda)$. Since $\tilde{T}(\lambda)$ is even it
follows that $\exists$ a positive integer $l$ such that
$\tilde{T}(\lambda)\thicksim\lambda^{2l}$ in a neghbourhood of
$\lambda=0$. Recall that $Q_{\delta}(\lambda)\neq 0$ on
$\textbf{a}^{*}$ and
$h(\lambda)=\frac{a_{j}(\lambda)}{Q_{\delta}(\lambda)}$ is even,
holomorphic. Now (\ref{eqn11}) implies that
\bea\label{eqn12}\int_{|\lambda|\leq
\varepsilon}\left|\frac{h(\lambda)}{\tilde{T}(\lambda)}\right|^{2}|c(\lambda)|^{-2}d\lambda<\infty.\eea
for some $\varepsilon>0$. Since
$|c(\lambda)|^{-2}\thicksim\lambda^{2}$ near zero (\ref{eqn12})
implies that $h(\lambda)=0$ if $\lambda=0$. Since $h(\lambda)$ is
even $h(\lambda)\thicksim\lambda^{2m}$ in a neighborhood of
$\lambda = 0 .$ Then  (\ref{eqn12}) implies that $m \geq l$ which
inturn implies that  $\frac{a_{j}(\lambda)}{\tilde{T}(\lambda)}$
is entire which is of exponential type by Malgrange's theorem.
This finishes the proof.

\textbf{(2)}: If $G$ has only one congugacy class of Cartan
subgroups then the Plancherel density $|c(\lambda)|^{-2}$ is
given by a polynomial which we describe now. Let $\sum_{0}^{+}$
be the set of positive indivisible roots. If
$\alpha\in\sum_{0}^{+}$ then the multiplicity $m_{\alpha}$ is
even $\forall\alpha$ and $m_{2\alpha}=0$. For
$\alpha\in\sum_{0}^{+}$ define $$\lambda_{\alpha}=\frac{\langle
\lambda,\alpha\rangle}{\langle\alpha,\alpha\rangle},\vspace*{0.1in}~~\lambda\in
\textbf{a}_{\mathbb{C}}^{*}.$$ With the convention that the
product over an empty set is $1$ the explicit expression for
$|c(\lambda)|^{-2}$ is given by
$$|c(\lambda)|^{-2}=c\prod_{\alpha\in\sum_{0}^{+}}\lambda_{\alpha}^{2}\prod_{k=1}^{m_{\alpha}/2-1}(\lambda_{\alpha}^{2}+k^{2}),$$
(see \cite{HP}) where $c$ is a positive constant.

Proceeding as in the previous case we obtain that
$$\tilde{f}(\lambda)=\frac{\tilde{g}(\lambda)}{\tilde{T}(\lambda)}.$$
Notice that both $\tilde{f}(\lambda)$ and $\tilde{g}(\lambda)$
belong to $Hom(V_{\delta},V_{\delta}^{M})$. Write
$\tilde{f}(\lambda)=((\tilde{f_{ij}}(\lambda)))$ and
$\tilde{g}(\lambda)=((\tilde{g_{ij}}(\lambda))).$ By the
Plancherel theorem we have
$$\int_{\textbf{a}^{*}}\left|\frac{g_{ij}(\lambda)}{\tilde{T}(\lambda)}\right|^{2}|c(\lambda)|^{-2}d\lambda<\infty.$$
From the above and the expression for $|c(\lambda)|^{-2}$ we also
have
\bea\label{eqn13}\int_{\textbf{a}^{*}}\left|\frac{p(\lambda)g_{ij}(\lambda)}{\tilde{T}(\lambda)}\right|^{2}d\lambda<\infty.\eea
where $p(\lambda)$ is the polynomial given by
$$p(\lambda)=\Pi_{\alpha\in\sum_{0}^{+}}\lambda_{\alpha}.$$ Let
$dim~~ \textbf{a}^{*}=l$. Since $\lambda\rightarrow
p(\lambda)g_{ij}(\lambda)$ is an entire function of exponential
type with rapid decay on $\textbf{a}^{*}$, we have $G\in
C_{c}^{\infty}(\mathbb{R}^l)$ such that the Euclidean Fourier
transform of $G$, $\hat{G}(\lambda)=p(\lambda)g_{ij}(\lambda).$
Similarly let $S$ be the compactly supported distribution on
$\mathbb{R}^{l}$ such that $\hat{S}(\lambda)=\tilde{T}(\lambda)$.
From (\ref{eqn13}) it follows that there exists $F\in
L^{2}(\mathbb{R}^{l})$ such that $F*_{\mathbb{R}^{l}}S=G$. Since
$\hat{S}(\lambda)=\tilde{T}(\lambda)$ satisfies the conditions in
Theorem \textbf{2.1} we obtain that $F\in
C_{c}^{\infty}(\mathbb{R}^{l})$. It follows that
$\frac{p(\lambda)g_{ij}(\lambda)}{\tilde{T}(\lambda)}$ is an
entire function of exponential type with rapid decay on
$\textbf{a}^{*}$. However we need to show that
$\frac{g_{ij}(\lambda)}{\tilde{T}(\lambda)}$ is entire. This
follows from applying the following lemma to matrix entries of
$\frac{Q_{\delta}(\lambda)^{-1}\tilde{g}(\lambda)}{\tilde{T}(\lambda)}$.
\end{proof}

\begin{lem}
Let $p(\lambda)$ be as above and $\psi(\lambda)$ be a holomorphic
function defined on
$\bf{a}_{\mathbb{C}}^{*}-\{\lambda:p(\lambda)=0\}$ such that
$p(\lambda)\psi(\lambda)$ has an entire extension. If
$\psi(\lambda)$ is Weyl group invariant then $\psi(\lambda)$ is an
entire function.
\end{lem}

\begin{proof}
Since $p(\lambda)$ is a product of irreducibles it suffices to
show that $R(\lambda)=p(\lambda)\psi(\lambda)$  vanishes on
$\{\lambda\in \textbf{a}_{\mathbb{C}}^{*}: p(\lambda)=0\}$. This
will follow if we show that $R(\lambda)$ vanishes on $\{\lambda\in
\bf{a}^{*}: p(\lambda)=0\}.$  Fix $\alpha\in\sum_{0}^{+}$ and let
$0\neq\lambda_{0}\in \textbf{a}^{*}$ be such that
$\langle\alpha,\lambda_{0}\rangle=0$ and $\langle \beta ,
\lambda_0 \rangle \neq 0$ if $\beta \neq \alpha.$   It is easy to
see that, in a small enough neighborhood of $\lambda_{0},$
$\langle\alpha,\lambda\rangle$ takes both positive and negative
values while Sgn $(\langle\beta,\lambda\rangle)$ is constant
$\forall \beta\in\sum_{0}^{+},$ $\beta\neq\alpha$. Since
$\psi(\lambda)$ is Weyl group invariant this will force
$R(\lambda)=0$ if $\lambda=\lambda_{0}$. This proves that
$R(\lambda)$ is zero on (real) $(n-1)$  dimensional strata of the
set $\{ \lambda \in \textbf{a}^{*}:~ p(\lambda) = 0 \}.$ This
clearly implies that $R(\lambda) = 0 $ whenever $p(\lambda) = 0 .$
This finishes the proof.
\end{proof}

Our proof works well for many other cases as well. To explain this
first we reproduce the analysis of $c$-function from \cite{HP}.
The Plancherel density $|c(\lambda)|^{-2}$ is given by the product
formula
$$|c(\lambda)|^{-2}=c\prod_{\alpha\in\sum_{0}^{+}}|c_{\alpha}(\lambda)|^{-2}.$$
Recall that if both $\alpha$ and $2\alpha$ are roots, then
$m_{\alpha}$ is even and $m_{2\alpha}$ is odd. Consider the
following cases :
\begin{itemize}
\item[\textbf{(a)}] $m_{\alpha}$ even, $m_{2\alpha}=0$
\item[\textbf{(b)}] $m_{\alpha}$ odd, $m_{2\alpha}=0$
\item[\textbf{(c)}] $m_{\alpha}/2$ even, $m_{2\alpha}$ odd
\item[\textbf{(d)}] $m_{\alpha}/2$ odd, $m_{2\alpha}$ odd.
\end{itemize}
If
$\lambda_{\alpha}=\frac{\langle\lambda,\alpha\rangle}{\langle\alpha,\alpha\rangle}$,
with the convention that product over an empty set is $1$, the
explicit expression for $|c_{\alpha}(\lambda)|^{-2}$ is given by
(upto a constant)
$\lambda_{\alpha}p_{\alpha}(\lambda)q_{\alpha}(\lambda)$ where
$p_{\alpha}$ and $q_{\alpha}$ in the four cases listed above are
the following:
\begin{itemize}
\item[\textbf{(a)}] $p_{\alpha}(\lambda)=\Pi_{k=1}^{\frac{m_{\alpha}}{2}-1}\left[\lambda_{\alpha}^{2}+k^{2}\right]$,  \item[~]$q_{\alpha}(\lambda)=1.$
\item[\textbf{(b)}]$p_{\alpha}(\lambda)=\Pi_{k=0}^{\frac{m_{\alpha}-3}{2}}\left[\lambda_{\alpha}^{2}+(k+\frac{1}{2})^{2}\right]$,  \item[~]$q_{\alpha}(\lambda)=tanh\pi\lambda_{\alpha}.$
\item[\textbf{(c)}]$p_{\alpha}(\lambda)=\Pi_{k=0}^{\frac{m_{\alpha}}{4}-1}\left[(\frac{\lambda_{\alpha}}{2})^{2}+(k+\frac{1}{2})^{2}\right]\Pi_{k=0}^{\frac{m_{\alpha}}{4}+\frac{m_{\alpha}-1}{2}-1}\left[(\frac{\lambda_{\alpha}}{2})^{2}+(k+\frac{1}{2})^{2}\right]$,  \item[~]$q_{\alpha}(\lambda)=tanh\frac{\pi\lambda_{\alpha}}{2}.$
\item[\textbf{(d)}]$p_{\alpha}(\lambda)=\Pi_{k=0}^{\frac{m_{\alpha}-2}{4}}\left[(\frac{\lambda_{\alpha}}{2})^{2}+k^{2}\right]\Pi_{k=1}^{\frac{m_{\alpha}+2m_{2\alpha}}{4}-1}\left[(\frac{\lambda_{\alpha}}{2})^{2}+k^{2}\right]$,  \item[~]$q_{\alpha}(\lambda)=coth\frac{\pi\lambda_{\alpha}}{2}$.
\end{itemize}
The case \textbf{(a)} corresponds to the case dealt with in
Theorem 3.9. It is clear from the above expression that if
$m_{\alpha}$ is large enough $\forall\alpha\in\sum_{0}^{+}$ then
$$\lambda_{\alpha}p_{\alpha}(\lambda)q_{\alpha}(\lambda)\geq\lambda_{\alpha}^{2},\vspace*{0.1in}~~\forall\alpha\in\sum_{0}^{+}$$
and consequently we obtain (3.3). Hence the theorem holds for all
groups with this property. Simple Lie groups with this property
can be read off from the list in \cite{W}.

\end{document}